\documentclass[11pt]{amsart}
\usepackage{geometry}
\usepackage{amscd,amssymb,verbatim,xcolor}
\usepackage{longtable, multirow,relsize}
\usepackage{soul,cancel}

\date{\today}

\newcommand\bb[1]{{\mathbb     #1}}
\newcommand\C[1]{{\mathcal #1 }}
\newcommand\fk[1]{\mathfrak{#1}}
\newcommand\ovl[1]{\overline{#1}}

\newcommand\la{{\lambda}}
\newcommand\ep{{\epsilon}}

\newcommand\fg{{\mathfrak g}}

\newcommand\eg{{e.g. ~}}

\newtheorem{theorem}{Theorem}[section]

\newtheorem{corollary}[theorem]{Corollary}

\newtheorem{definition}[theorem]{Definition}
\newtheorem{example}[theorem]{Example}
\newtheorem{lemma}[theorem]{Lemma}
\newtheorem{proposition}[theorem]{Proposition}
\newtheorem{remark}[theorem]{Remark}

\newcommand\Ind{{\operatorname{Ind}}}

\begin{document}
\title{Admissible modules and normality of classical nilpotent orbits II}
\author{Dan Barbasch}
\author{Kayue Daniel Wong}

\address[Barbasch]{Department of Mathematics, Cornell University, Ithaca, NY 14853,
U.S.A.}
\email{barbasch@math.cornell.edu}

\address[Wong]{School of Science and Engineering, The Chinese University of Hong Kong, Shenzhen,
Guangdong 518172, P. R. China}
\email{kayue.wong@gmail.com}

\begin{abstract}
As a sequel to \cite{BW1}, we study the character formula of the Brylinski model $\C B(\ovl{\C O})$ for
classical nilpotent varieties $\ovl{\C O}$. As a consequence, one can compute the multiplicities of
all $K-$types of the ring of regular functions $R(\ovl{\C O})$ for all classical nilpotent varieties.
\end{abstract}

\maketitle
\setcounter{tocdepth}{1}

\section{Introduction}
Let $G = Sp(2n,\mathbb{C})$ or $SO(2n+\delta,\mathbb{C})$ ($\delta \in \{0,1\}$) be a complex classical
Lie group. 
In \cite{Br}, Brylinski defined a $(\fk g, K)-$module $\C B(\ovl{\C O})$ for all classical nilpotent orbits, 
whose $G\cong K_c-$spectrum (here the subscript $c$ denotes complexification) satisfies
$$\C B(\ovl{\C O})|_{K_c} \cong R(\ovl{\C O})|_G,$$
where $R(\ovl{\C O})$ is the ring of regular functions of $\ovl{\C O}$.
In \cite{BW1}, it is proved that $\C B(\ovl{\C O})$ 
is the cyclic submodule of a `deformed' unipotent representation $\C B(\C O)$
attached to $\C O$ satisfying $\C B(\C O)|_{K_c} \cong R(\C O)$ (see \eqref{eq:sunip} below).
Moreover, the diminutive $K_c-$type multiplicities of $\C B(\ovl{\C O})$ are obtained in \cite{BW1},
resulting in an alternative proof on the classification of (non-)normal classical nilpotent varieties in \cite{KP}.

\medskip
In this paper, we obtain the character formula of $\C B(\ovl{\C O})$. 
As a result, the multiplicities of all $G\cong K_c-$types of $R(\ovl{\C O})$ can be effectively computed.
An important tool of our study of $\C B(\ovl{\C O})$ is a reduction theorem, which 
reduces the problem of studying to $\C B(\ovl{\C O})$ to that 
of $\C B(\ovl{\C O_{gen}})$, where $\C O_{gen}$ is the {\it generic part} of $\C O$ (Theorem \ref{thm:reduce}).

\subsection{Notations}
We introduce the notations and terminologies needed for this paper. Many of them
are introduced in \cite{BW1}:

\smallskip
\noindent {\bf (a) Irreducible representations for complex Lie groups.} Let $G$ be a complex Lie group viewed as a real Lie
group with Cartan involution $\theta$. Consider $H=T\cdot A$
be the Cartan decomposition of the Cartan subgroup $H$ of $G$.
Then the {\it Langlands parameter} of any irreducible module 
$(\fg_c, K_c)-$modules is a pair $(\la_L;\la_R)$
 such that $\mu:=\la_L-\la_R$ is the parameter of a character of $T$, and $\nu:=\la_L+\la_R$ the
  $A-$character. 
	
The \textit{principal series representation}
associated to $(\la_L;\la_R)$ is the $(\mathfrak{g}_c, K_c)-$module
$$X(\lambda_L,\lambda_R) =\Ind_B^G(e^{\mu}\otimes e^\nu \otimes
1)_{K-finite},$$
where the symbol $\Ind$ refers to Harish-Chandra induction, and let
$\ovl{X}(\la_L;\la_R)$ be the unique irreducible subquotient of
$X(\la_L;\la_R)$ containing  the $K_c-$type with extremal weight
$\mu.$ This is called the {\it Langlands subquotient}. 
The infinitesimal character of $X(\la_L;\la_R)$ and $\ovl{X}(\la_L;\la_R)$, when $\fk g_c$ is identified with $\fk
g\times\fk g,$ is $(\la_L;\la_R).$ 

\smallskip
The case of the orthogonal group is dealt with via Clifford theory. We use Weyl's parametrization of the
finite dimensional representations of the orthogonal groups (see page 6 of \cite{AB}). The highest weight of a $K_c-$type 
will be denoted 
\begin{equation} \label{eq:weylo}
\mu = (a_1 \geq \dots \geq a_n|\pm 1)
\end{equation} 
whenever $a_n = 0$, and the Langlands quotients will acquire a $\pm$ whenever the
corresponding lowest $K_c-$types are in different irreducible quotients.

\smallskip
If we need to specify the group, the standard module and Langlands
quotient will acquire a subscript, \eg $X_G(\la_L;\la_R)$ or $\ovl{X}_G(\la_L;\la_R)$ 
(respectively $X_G(\la_L;\la_R| \pm 1)$ or $\ovl{X}_G(\la_L;\la_R|\pm 1)$ for orthogonal groups).

\medskip 

\noindent {\bf (b) Shorthand for parabolically induced modules.}
Let $G'$ be the Lie group with the same
type as $G$ of lower rank, and $\Psi$ be a representation of 
$G'$. We write
\[I^G\left(
  \begin{pmatrix}\la_L\\ \la_R\end{pmatrix} \boxtimes\cdots \boxtimes \begin{pmatrix} \la_L'\\ \la_R' \end{pmatrix}\boxtimes\Psi\right) :=
\mathrm{Ind}_{GL \times \cdots \times GL \times G'}^{G}\left(\ovl{X}_{GL}(\la_L;\la_R)
\boxtimes \cdots \boxtimes \ovl{X}_{GL}(
  \la_L';\la_R') \boxtimes \Psi\right).
\]
Similarly, we write 
\[
I^G\left(
   \begin{pmatrix}\la_L\\ \la_R\end{pmatrix} \boxtimes\cdots \boxtimes \begin{pmatrix} \la_L'\\ \la_R' \end{pmatrix} \right) :=
\mathrm{Ind}_{GL \times \cdots \times GL}^{G}\left(\ovl{X}_{GL}(\la_L;\la_R) 
\boxtimes \cdots \boxtimes \ovl{X}_{GL}(\la_L';\la_R')\right).
\]

\medskip

\noindent {\bf (c) Diminutive $K_c-$types.}
Let $G$ be a classical complex Lie group, and $V_{\mu}$ be the irreducible, finite-dimensional $K_c-$type with highest weight $\mu$ (or Weyl's parametrization \eqref{eq:weylo} for orthogonal groups). The
{\bf diminutive $K_c-$types} of $G$ are:
$$\begin{cases}
V_{(1^{k},0^{n-2k},-1^{k})}\ \  ( 0 \leq k \leq \lfloor \frac{n}{2} \rfloor) & \text{for}\ G = GL(n,\mathbb{C}) \\
V_{(1^{k},0^{n-k}|\ (-1)^k)}\ \ ( 0 \leq k \leq n)  & \text{for}\ G = O(2n+1,\mathbb{C})\\
V_{(1^{2k},0^{n-2k})}\ \ ( 0 \leq k \leq \lfloor \frac{n}{2} \rfloor) & \text{for}\ G = Sp(2n,\mathbb{C}) \\
V_{(1^{2k},0^{n-2k}|\ \pm 1)}\ \ ( 0 \leq k \leq \lfloor \frac{n}{2} \rfloor) \  & \text{for}\ G = O(2n,\mathbb{C})
\end{cases}.$$ 
In particular, the diminutive $K_c-$types are equal to $\wedge^{2\ell}\mathbb{C}^{2n+\delta}$ ($\delta \in \{0,1\}$) 
for orthogonal groups, and $\wedge^{2\ell}\mathbb{C}^{2n}/\wedge^{2\ell-1}\mathbb{C}^{2n+\delta}$ for symplectic groups.

Let $X$, $Y$ be two admissible $(\fg_c, K_c)-$modules. We write
$$X \ \approx\  Y$$ 
if $X$ and $Y$ have the same composition factors with diminutive lowest $K_c-$types (with multiplicities).
Note that this is a weaker notion than $\stackrel{dm}{\rightarrow}$ defined in \cite{BW1}.

\medskip

\noindent {\bf (d) Shorthand for Langlands parameters.} Let $a, A \in \frac{1}{2}\mathbb{Z}$ be such that $A - a \in \mathbb{Z}$. 
A {\it string} is an ascending sequence of numbers $(a \ldots A) := (a,a+1,\dots,A-1,A)$
(if $A - a < 0$ then the string is empty). 

For $x, y \in \mathbb{Z}_{\geq -1}$ of the same parity, define $\la[x,y]^+$ by the spherical Langlands parameter
\begin{equation} \label{eq:laxy}
\la[x,y]^+ := \begin{pmatrix} -\frac{y}{2} +1 \ldots \frac{x}{2}\\ -\frac{y}{2} +1 \ldots \frac{x}{2} \end{pmatrix}
\end{equation}
If $x \geq y$, we denote by $\la[x,y]^-$ the {\bf non-spherical} parameter 
$$
  \la[x,y]^- := \begin{pmatrix} -\frac{y}{2}+1 \ldots \frac{y}{2},  &\frac{y}{2}+1 \ldots \frac{x}{2} \\ 
	-\frac{y}{2} \ldots \frac{y}{2}-1, &\frac{y}{2}+1 \ldots \frac{x}{2} \end{pmatrix}.
$$
Note that 
\begin{equation} \label{eq:laxy2}
\begin{aligned}
\ovl{X}_{GL}(\la[x,y]^+) &= |\det|^{x-y},\\ 
\ovl{X}_{GL}(\la[x,y]^-) &= \mathrm{Ind}^{GL(\frac{x+y}{2})}_{GL(y) \times GL(\frac{x-y}{2})}(\det \boxtimes |\det|^y)\end{aligned}
\end{equation}
\medskip

\noindent {\bf (e) Nilpotent orbits.} A classical nilpotent orbit $\C O$ is denoted by the column sizes of its corresponding partition
\begin{equation} \label{eq:bcdo}
\C O = \begin{cases} 
(c_0 \geq c_1 \geq \dots \geq c_{2p} \geq c_{2p+1}) & \text{if}\ G = Sp(2n,\mathbb{C}), \\
(c_1 \geq c_2 \geq \dots \geq c_{2p} \geq c_{2p+1}) & \text{if}\ G = O(2n+\delta,\mathbb{C})\ (\delta \in \{0,1\})
\end{cases},
\end{equation}
such that $c_{2i-1} + c_{2i}$ is an even integer for all $i$. We say $\C O$ is even or odd if $c_i$ are all even or odd integers respectively.

Let $\tau(\C O) := \{i\ |\ c_{2i-1} = c_{2i}\}$. Define $\C O'$ by
$$
\C O' := \begin{cases} (c_0' \geq c_1' \geq \dots \geq c_{2q}' \geq c_{2q+1}') & \text{if}\ G' = Sp(2n',\mathbb{C}), \\
(c_1' \geq c_2' \geq \dots \geq c_{2q}' \geq c_{2q+1}')& \text{if}\ G' = O(2n'+\delta,\mathbb{C})
\end{cases}$$
where $\C O'$ is obtained from $\C O$ by removing all $c_{2i}, c_{2i+1}$ for $i \in \tau(\C O)$ 
in $\C O$. Then $\C O'$ is a nilpotent orbit in $\mathfrak{g}'$ of the same type as $\fg$ but of smaller rank. 

For example, let $\C O = (9,{\it 9,9},9,8,{\it 6,6},6,5,5,4,{\it 2,2,2,2},0)$ be a symplectic nilpotent orbit, then
$\tau(\C O) = \{1,3,6,7\}$ and $\C O' = (9,9,8,6,5,5,4,0).$

Using the above notations, 
\begin{equation} \label{eq:sunip}
\C B(\C O) := \mathrm{Ind}_{\prod_{i \in \tau(\C O)} GL(c_{2i}) \times G'}^{G}\left(|\det| \boxtimes \dots \boxtimes |\det| \boxtimes
\mathcal{U}(\mathcal{O}')\right),
\end{equation}
where $\mathcal{U}(\mathcal{O}')$ is the {\it spherical unipotent representation} attached to a classical nilpotent orbit $\C O'$
given by
$$\mathcal{U}(\mathcal{O}') := \begin{cases}
\overline{X}_{G'}\left(\la[c_0',c_1']^+, \dots,  \la[c_{2q},c_{2q+1}']^+\right) &\text{if}\ G' = Sp(2n',\mathbb{C})\\
\overline{X}_{G'}\left(\la[-\epsilon,c_1']^+, \la[c_2',c_3']^+, \dots,  \la[c_{2q},c_{2q+1}']^+ | +1\right) &\text{if}\ G' = O(2n'+\delta,\mathbb{C})\end{cases}$$

We now state the main results of \cite{BW1}:
\begin{theorem} \label{thm:1}
Let $\C O$ be as given in \eqref{eq:bcdo}. The Brylinski model $\C B(\ovl{\C O})$ in \cite{Br} is the cyclic submodule of $\C B(\C O)$ defined in \eqref{eq:sunip}.
\end{theorem}
To compute the diminutive $K_c-$type multiplicities of $\C B(\ovl{\C O})$, one has
\begin{theorem} \label{thm:2}
Let $\C O$ be as given in \eqref{eq:bcdo}, and
\begin{equation} \label{eq:gammao}
\Gamma(\C O) := \begin{cases}
I^G\left(\la[c_0,c_1]^+\boxtimes\cdots\boxtimes\la[c_{2p},c_{2p+1}]^+\right) & \text{if}\ G = Sp(2n,\mathbb{C}),   \\
I^G\left(\la[c_2,c_3]^+\boxtimes\cdots\boxtimes\la[c_{2p},c_{2p+1}]^+\boxtimes\ \mathrm{T}(c_1|+1) \right) & \text{if}\ G = O(2n+\delta,\mathbb{C})
\end{cases}
\end{equation}
where $\mathrm{T}(x|+1)$ and $\mathrm{T}(x|-1)$ are the trivial and sign representation of $O(x,\mathbb{C})$ respectively.
Then 
$$\Gamma(\C O) \approx \C B(\ovl{\C O}).$$
\end{theorem}
As proved in the Appendix B of \cite{BW1}, we also record the following, which will be studied in 
full detail in Section \ref{sec:factor} below:
\begin{theorem} \label{thm:3}
The lowest $K_c-$types of the composition factors of $\C B(\ovl{\C O})$ and $\C B(\C O)$ are diminutive.
\end{theorem}

\section{Reduction to Generic Orbits}
In this section, we define the notion of generic orbits, and prove that our study of $\C B(\ovl{\C O})$ 
can be reduced to that of generic orbits.

\begin{definition}
Let $\C O$ be a classical nilpotent orbit given in \eqref{eq:bcdo}. We say $\C O$ is {\bf non-generic} if there exists $c_i = c_{i+1} = c_{i+2}$ for some $i$. Otherwise, it is called {\bf generic}. 
In other words, $\C O$ is generic if and only if the number of repeated columns of $\C O$ is at most two.
\end{definition}
Let $N_{\mathcal{O}}(\ell):= \#\{i\ |\ c_i = \ell\}$ be the number of columns of $\C O$ with size equal to $\ell$.
The {\it generic part} $\C O_{gen}$ of $\C O$ is defined by the following: 
For each $\ell \in \mathbb{N}$,
$$N_{\mathcal{O}_{gen}}(\ell) = \begin{cases} 
1 & \text{if}\ N_{\mathcal{O}}(\ell) > 2\ \text{is odd}; \\
2 & \text{if}\ N_{\mathcal{O}}(\ell) > 2\ \text{is even};\\
N_{\mathcal{O}}(\ell) &  \quad \quad \text{otherwise}
\end{cases}$$
That is, $\C O_{gen}$ is obtained from $\C O$ by removing as few column pairs $c_{2i-1} = c_{2i}$ from $\C O$ as possible
such that $\C O_{gen}$ is generic.

For example, for the symplectic orbit $\C O = (9,\underline{9,9},9,8,\underline{6,6},6,5,5,4,\underline{2,2},2,2,0)$ above, 
$$\C O = \underline{(9,9)} \cup \underline{(6,6)} \cup \underline{(2,2)} \cup (9,9,8,6,5,5,4,2,2,0),$$
i.e. $\C O_{gen} = (9,9,8,6,5,5,4,2,2,0)$.
Note that $\C O_{gen}' = \C O'$, both equal to $(9,9,8,6,5,5,4,0)$.

\begin{theorem} \label{thm:reduce}
Let $\C O$ be classical nilpotent given in \eqref{eq:bcdo} with generic part $\C O_{gen}$. Suppose
$\C O = (g_1, g_1) \cup \dots \cup (g_k,g_k) \cup \C O_{gen}$.
Then
$$\C B(\ovl{\C O}) \cong I^G\left(\la[g_1,g_1]^+ \boxtimes \cdots \ \boxtimes \la[g_k,g_k]^+ \boxtimes\ \C B(\ovl{\C O_{gen}})\right).$$
\end{theorem}
\begin{proof}
By Theorem \ref{thm:1}, $\C B(\ovl{\C O_{gen}}) \subseteq \C B(\C O_{gen})$ as a cyclic submodule. Using induction in stages, we have
\begin{align*}
&I^G\left(\la[g_1,g_1]^+ \boxtimes \cdots \ \boxtimes \la[g_k,g_k]^+ \boxtimes \C B(\ovl{\C O_{gen}})\right) \\
\subseteq\ &I^G\left(\la[g_1,g_1]^+ \boxtimes \cdots \ \boxtimes \la[g_k,g_k]^+ \boxtimes\ \C B(\C O_{gen})\right) \cong \C B(\C O)
\end{align*}
The last $\cong$ holds since $\C O_{gen}' = \C O'$, and all $\la[g_i,g_i]^+$ comes from $\la[c_{2i}, c_{2i-1}]^+$ for some $i \in \tau(\C O)$.
Therefore, we have the inclusion
\begin{equation} \label{eq:genincl}
\C B(\ovl{\C O}) \subseteq I^G\left(\la[g_1,g_1]^+ \boxtimes \cdots \ \boxtimes \la[g_k,g_k]^+ \boxtimes\ \C B(\ovl{\C O_{gen}})\right) \quad (\subseteq \C B(\C O)).
\end{equation}
since $\C B(\ovl{\C O})$ is a cyclic submodule of $\C B(\C O)$. 

To see the above inclusion is an equality, note that $\Gamma(\C O_{gen}) \approx \C B(\ovl{\C O_{gen}})$ by Theorem \ref{thm:2} and hence
$$\Gamma(\C O) = I^G\left(\la[g_1,g_1]^+ \boxtimes \cdots \ \boxtimes \la[g_k,g_k]^+ \boxtimes \Gamma(\C O_{gen})\right) \approx I^G\left(\la[g_1,g_1]^+ \boxtimes \cdots \ \boxtimes \la[g_k,g_k]^+ \boxtimes\ \C B(\ovl{\C O_{gen}})\right)$$
On the other hand, we have $\Gamma(\C O) \approx \C B(\ovl{\C O})$ by Theorem \ref{thm:2} again. 
Consequently, 
$$\C B(\ovl{\C O}) \approx I^G\left(\la[g_1,g_1]^+ \boxtimes \cdots \ \boxtimes \la[g_k,g_k]^+ \boxtimes\ \C B(\ovl{\C O_{gen}})\right)$$
have the same diminutive $K_c-$type multiplicities. 

As a result, the inclusion \eqref{eq:genincl} is an isomorphism on the level of diminutive $K_c-$types.
Now the result follows from Theorem \ref{thm:3}, since the lowest $K_c-$types of all composition factors of both modules of \eqref{eq:genincl} are diminutive.
\end{proof}

From now on, we assume our orbit $\C O = \C O_{gen}$ is generic, which has at most two repeated columns.

\section{Composition factors of $\C B(\ovl{\C O})$} \label{sec:factor}
In Appendix B of \cite{BW1}, a list of candidates of the composition factors of $\C B(\ovl{\C O})$
(and $\C B(\C O)$) is given. We recall them as follows:

\begin{definition} \label{def:funddegeneration}
Let $(b_0 \geq b_1 = b_2 \geq b_3)$ be four positive integers of the same parity. 
A {\bf fundamental degeneration} is defined by:

$\bullet\ (b_0>b_1=b_2>b_3) \rightarrow (b_0, b_1+2, b_2-2, b_3)$.

$\bullet\ (b_0=b_1=b_2>b_3)  \rightarrow (b_0+1, b_1+1, b_2-2, b_3)$.

$\bullet\ (b_0>b_1=b_2=b_3)  \rightarrow (b_0, b_1+2, b_2-1, b_3-1)$.

$\bullet\ (b_0=b_1=b_2=b_3)  \rightarrow (b_0+1, b_1+1, b_2-1, b_3-1)$.
\end{definition}
In the first two cases, we omit the columns $b_2-2,b_3$ 
if both terms are equal to zero. Note that when $b_0 > b_1$, 
the size of $b_0$ remains unchanged after degeneration. Similarly, 
if $b_2 > b_3$, the size of $b_3$ is the same after degeneration.

\begin{definition} \label{def:normo}
Let $\C O$ be a classical orbit given in \eqref{eq:bcdo} (take $c_0 = \infty$ for orthogonal orbits). We construct a collection of orbits as follows:

\medskip
\noindent (1)\ For each $\dots c_{2i} \geq c_{2i+1} = c_{2i+2} = \dots = c_{2j-1} = c_{2j} \geq c_{2j+1} \dots$ 
appearing in $\C O$, perform fundamental degeneration on the columns 
$c_{2i} \geq c_{2i+1} = c_{2j} \geq c_{2j+1}$
and get a new orbit:
$$(c_0 \geq \dots \geq c_{2i}'\geq c_{2i+1}' \geq c_{2i+2} = \dots = c_{2j-1}\geq c_{2j}'\geq c_{2j+1}'\geq \dots\geq c_{2p+1}).$$
\noindent (2)\ For each new orbit obtained in Step (1), repeat Step (1) on them until
  there are no more $c_{2j+1} = c_{2j+2}$'s.
		
\medskip
Denote the collection of all such orbits by $Norm(\C O)$. 
\end{definition}
Note that from the definition, $\C P \in Norm(\C O)$ may not be generic. This can also be seen
in the example below:
\begin{example}
Let $\C O = (8, 6, 6, 4, 4, 2, 2, 0)$ be a symplectic orbit. Then the $\C P \in Norm(\C O)$ are given by
\begin{equation} \label{eq:8664422}
  \begin{matrix}
    & &(8 6  6 4  4 2  2  0)&\\
		&&&\\
&(8 8 4 4  4 2  2  0)&(8 6  6  6 2 2  2  0)&(8 6  6 4  4  4)\\
&&&\\
&(8 8 5  5 2 2  2  0)&(8  8 4 4  4  4)&(8 6  6  6 3  3)\\
&&&\\
&         &(8  8 5  5 3  3)
  \end{matrix}.
	\end{equation}
\end{example}

We now list the candidates of the composition factors of $\C B(\ovl{\C O})$.
For simplicity, we only study the case when $\C O$ is {\bf even, generic}. 
In such a case, each $\C P \in Norm(\C O)$ is a special orbit in the sense of Lusztig.

\smallskip
The following Lemma is essential in the determination of the candidates of composition factors:
\begin{definition} \label{def:pprime}
Let $\C O$ be an {\bf even, generic} orbit (so that $\delta = 0$ for orthogonal groups). For each 
$${\C P} := \begin{cases} (d_0 \geq d_1 \geq \dots \geq d_{2s} \geq d_{2s+1}) & \text{if}\ G = Sp(2n,\mathbb{C}),\\
(d_1 \geq d_2 \geq \dots \geq d_{2s} \geq d_{2s+1}) & \text{if}\ G = O(2n,\mathbb{C})
\end{cases}$$
in $Norm(\C O)$, let $\tau_0(\C P) := \{i\ |\ d_{2i-1} = d_{2i}\ \text{is even}\} = \{i_1, \dots, i_r\}$ and 
$${\C P^*} := \begin{cases} (d_0^* \geq d_1^* \geq \dots \geq d_{2t}^* \geq d_{2t+1}^*) & \text{if}\ G^* = Sp(2n^*,\mathbb{C}),\\
(d_1^* \geq d_2^* \geq \dots \geq d_{2t}^* \geq d_{2t+1}^*)  & \text{if}\ G^* = O(2n^*,\mathbb{C})
\end{cases}$$
is obtained from $\C P$ by removing the columns $d_{2i-1} = d_{2i}$ for $i \in \tau_0(\C P)$. 
The {\bf parameter group} of $\C P$ is given by
$${\C A}(\C P) \cong (\mathbb{Z}/2\mathbb{Z})^{a_0} \times \dots \times (\mathbb{Z}/2\mathbb{Z})^{a_{t}}, \quad a_i := \begin{cases} 1 & \text{if}\ d_{2i+1}^*\ \text{is even and positive}, \\ 0 & \text{if}\ d_{2i+1}^*\ \text{is odd or}\ = 0 \end{cases}.$$
Consequently, all irreducible representations $\ep \in \C A(\C P)^{\vee}$ of the parameter group is
$$\ep = (\ep_1, \dots, \ep_{t+1}), \quad \ep_i \in \begin{cases} \{+,-\} & \text{if}\ a_i =1,\\
\{+\} & \text{if}\ a_i = 0.\end{cases}$$
\end{definition}
When $\C P$ is a symplectic orbit, $\C A(\C P) \cong \ovl{A}(\C P)$ is isomorphic to the Lusztig's quotient group of $\C P$ (as well as $\C P^*$). On the other hand, if $\C P$ is an orthogonal orbit,
then $\C A(\C P) \cong \ovl{A}(\C P) \times \mathbb{Z}/2\mathbb{Z}$. The extra $\mathbb{Z}/2\mathbb{Z}$
is due to the fact that $O(2n,\mathbb{C})$ has two connected components.

\begin{proposition} \label{prop:parastructure}
Retain the notations of Definition \ref{def:pprime}. For each $\C P \in Norm(\C O)$ and 
$\epsilon \in \C A(\C P)^{\vee}$, consider the induced modules 
\begin{equation}\label{eq:parameter}
\C M(\C P, \ep) :=\ I^G\left( \la[d_{2i_1},d_{2i_1-1}]^-\boxtimes \cdots \boxtimes \la[d_{2i_r},d_{2i_r-1}]^-\boxtimes\ \mathcal{U}(\mathcal{P}^*;\epsilon)\right) \quad 
\end{equation}
where $\mathcal{U}(\mathcal{P}^*;\epsilon)$ are the special unipotent representations attached to the special orbit $\C P^*$ given by
 $$\mathcal{U}(\C P^*, \ep) := \begin{cases} \ovl{X}\left(\la[d_0^*,d_1^*]^{\ep_0}, \cdots, \la[d_{2t}^*,d_{2t+1}^*]^{\ep_{t}}\right) & \text{if}\ G^* = Sp(2n^*,\mathbb{C}),\\ 
\ovl{X}\left(\la[d_2^*,d_3^*]^{\ep_1}, \cdots, \la[d_{2t}^*,d_{2t+1}^*]^{\ep_{t}}, \la[0,d_1^*]^+| {\ep_0}\right) & \text{if}\ G^* = O(2n^*,\mathbb{C})\end{cases}$$
where $\la[x,y]^{\pm}$ are as defined in \eqref{eq:laxy} if $x$ and $y$ are even, and 
$$\la[x,y]^{+} := \left(-\frac{x-1}{2} \ldots \frac{x-1}{2}\ ;\ -\frac{x-1}{2} \ldots \frac{x-1}{2}\right)$$ 
if $x=y$ is odd.

Then each $\C M(\C P, \ep)$ is an irreducible module with diminutive lowest $K_c-$type, and
$$\{\C M(\C P, \ep)\ |\ \epsilon \in \C A (\C P)^{\vee}\}$$
are the candidates of the composition factors of $\C B(\ovl{\C O})$.
\end{proposition}

\begin{example}
Let $\C O = (8, 6, 6, 4, 4, 2, 2, 0)$ as in the previous example. For each $\C P \in Norm(\C O)$,
the columns ${\bf d_{2i-1} = d_{2i}}$ for $i \in \tau_0(\C P)$ are in {\bf bold}, so that the unbolded 
columns gives $\C P^*$. Then the 17 induced modules $\C M(\C P,\ep)$ in Equation \eqref{eq:parameter} are:
{\footnotesize
\[
  \begin{matrix}
    & &\begin{matrix}\C M(8{\bf 664422}0;+) = \\ 
		I^G({\bf \la[66]^-\boxtimes\la[44]^-\boxtimes\la[22]^-}\boxtimes\ \mathcal{U}(80;+)) \end{matrix}&\\
		&&&\\
		&&&\\
&\begin{matrix}\C M(884{\bf 4422}0;\pm+) =\\ 
I^G({\bf \la[44]^-\boxtimes\la[22]^-} \boxtimes \mathcal{U}(8840;\pm+))\end{matrix}&
\begin{matrix}\C M(8{\bf 66}62{\bf 22}0;\pm+) =\\ 
I^G({\bf \la[66]^-\boxtimes\la[22]^-} \boxtimes \mathcal{U}(8620;\pm+))\end{matrix}&
\begin{matrix}\C M(8{\bf 6644}4;\pm) =\\ 
I^G({\bf \la[66]^-\boxtimes\la[44]^-}\boxtimes \mathcal{U}(84;\pm))\end{matrix}\\
&&&\\
&&&\\
&
\begin{matrix}\C M(88552{\bf 22}0;\pm++) =\\ 
I^G({\bf \la[22]^-} \boxtimes \mathcal{U}(885520;\pm++))\end{matrix}&
\begin{matrix}\C M(884{\bf 44}4;\pm\pm) =\\ 
I^G({\bf \la[44]^-} \boxtimes \mathcal{U}(8844;\pm\pm))\end{matrix}&
\begin{matrix}\C M(8{\bf 66}633;\pm+) =\\ 
I^G({\bf \la[66]^-} \boxtimes \mathcal{U}(8633;\pm+))\end{matrix}\\
&&&\\
&&&\\
&         &\begin{matrix}\C M(885533;\pm++) =\\ 
\mathcal{U}(885533;\pm++)\end{matrix}
  \end{matrix}.
\]}
\end{example}

\begin{remark} \label{rmk:parastructure}
The $\C M(\C P, \ep)$ in the proposition above satisfies the following:

\noindent For $G = Sp(2n,\bb{C})$,
\begin{equation}
\C M(\C P, \ep) \approx 
I^G\left( \la[d_{2i_1},d_{2i_1-1}]^-\boxtimes \cdots \boxtimes \la[d_{2i_r},d_{2i_r-1}]^-\boxtimes \la[d_0^*,d_1^*]^{\ep_0} \boxtimes \cdots \boxtimes \la[d_{2t}^*,d_{2t+1}^*]^{\ep_{t}}\right).
\end{equation} 

\noindent For $G = O(2n,\bb{C})$,
\begin{equation}
\C M(\C P, \ep) \approx  I^G \begin{pmatrix}\la[d_{2i_1},d_{2i_1-1}]^-\boxtimes \cdots \boxtimes \la[d_{2i_r},d_{2i_r-1}]^- \boxtimes \quad \quad \quad \quad \\ 
\quad \quad \quad \la[d_2^*,d_3^*]^{\ep_1} \boxtimes \cdots \boxtimes \la[d_{2t}^*,d_{2t+1}^*]^{\ep_{t}} \boxtimes\ \mathrm{T}(d_1^*|\ep_0)\end{pmatrix}.
\end{equation} 
This can be seen by the calculations of the diminutive $K_c-$types of any
unipotent representations of classical groups in \cite{Wo3}, or Equation (16) of \cite{BW1}. For example,

\begin{align*}
\C M(88552{\bf 22}0;-++) &=I^G({\bf \la[2,2]^-} \boxtimes\ \mathcal{U}(885520;-++))\\
&\approx I^G\left(\la[2,2]^-\boxtimes\la[8,8]^-\boxtimes \la[5,5]^+\boxtimes \la[2,0]^+ \right).\\
&= \mathrm{Ind}_{GL(2)\times GL(8) \times GL(5) \times GL(1)}^{Sp(32)}\left(\det  \boxtimes \det \boxtimes \mathrm{triv} \boxtimes |\det| \right).
\end{align*}
where the last $=$ is given by \eqref{eq:laxy2} (in fact the $\approx$ here
happens also to be $=$, but we do not need this fact).
\end{remark}

\section{Distinguished modules} \label{sec:distinguished}
In the previous section, we have defined the set $Norm(\C O)$ and listed the parameters 
attached to each $\C P \in Norm(\C O)$ for all even, generic $\C O$. 
They are the candidates of the parameters of the composition 
factors of $\C B(\ovl{\C O})$.
We now define a distinguished parameter 
for each $\C P \in Norm(\C O)$. It will be proved 
in Theorem \ref{thm:main} that they are precisely the composition factors of $\C B(\ovl{\C O})$.

\begin{definition} \label{def:xbaro}
Let $\C O$ be an even, generic nilpotent orbit. For each $\C P \in Norm(\C O)$ with $\tau_0(\C P)$ and $\C P^*$ defined as in Definition \ref{def:pprime}, we say 
$$\ep(\C P) := \left(\ep(\C P)_0, \dots, \ep(\C P)_{t}\right) \in \C A(\C P)^{\vee}$$ 
is {\bf distinguished} in $\C P$ if the following holds
for all $i = 0, \dots, t$:
\begin{center}
If $d_{2i+1}^*$ even and positive, then $\ep(\C P)_i = (-1)^{\#\{j \in \tau_0(\C P) |\ d_{2i}^* \geq d_{2j} \geq d_{2i+1}^*\}}$
\end{center}
(take $d_0^* = \infty$ for orthogonal orbits). The irreducible module $\C M(\C P, \ep(\C P))$
is {\bf distinguished} in $\C P \in Norm(\C O)$.
\end{definition}
One can check directly from the definition that 
each $\C P \in Norm(\C O)$ contains exactly one
distinguished module. 

\begin{example}
Let $\C O = (8\geq 6\geq 6\geq 4\geq 4\geq 2\geq 2\geq 0)$ as above. Then the distinguished modules are:
{\small
\[
  \begin{matrix}
    & &\C M(8{\bf 664422}0;+)&\\
		&&&\\
&\C M(884{\bf 4422}0;++)&
\C M(8{\bf 66}62{\bf 22}0;-+)&
\C M(8{\bf 6644}4;+)\\
&&&\\
&
\C M(88552{\bf 22}0;+++)&
\C M(884{\bf 44}4;+-)&
\C M(8{\bf 66}633;-+)\\
&&&\\
&         &\C M(885533;+++)
  \end{matrix}.
\]}

For example, take $\C P = (d_0,d_1, d_2,{\bf d_3, d_4}, d_5) = (8,8,4,{\bf 4,4},4) \in Norm(\C O)$ with 
\begin{itemize}
\item $\tau_0(\C P) = \{2\}$ with $d_3 = d_4 = 4$.
\item $\C P^* = (d_0^*, d_1^*, d_2^*,d_3^*) = (8,8,4,4)$ with $\C A(\C P)^{\vee} = \{++, +-, -+,--\}$.
\end{itemize}
Now there are no $d_{2j}$, $j \in \tau_0(\C P)$ lying between $d_0^*$ and $d_1^* = 8$. Hence the sign is $+$. On the other hand, $d_4 = 4$ lies in between $d_2^*$ and $d_3^* = 4$. Hence
the sign is $-$. So the module $\C M(884444;+-) = I^G(\la[4,4]^- \boxtimes \mathcal{U}(8844;+-))$ is distinguished.
\end{example}

\begin{theorem} \label{thm:main}
Let $\C O$ be an even, generic nilpotent orbit. Then the sum of all distingished modules satisfies:
$$\C B(\ovl{\C O}) \approx \Gamma(\C O) \approx \bigoplus_{\C P \in Norm(\C O)} \C M(\C P, \ep(\C P)).$$
\end{theorem}

Since the composition factors of $\C B(\ovl{\C O})$ are all of the form $\C M(\C P,\ep)$ in Proposition \ref{prop:parastructure}. The above theorem implies:
\begin{corollary}
Retain the above notations. Then the composition factors of $\C B(\ovl{\C O})$ are precisely
$$\{ \C M(\C P, \ep(\C P))\ |\ \C P \in Norm(\C O)\},$$
all appearing with multiplicity one.
\end{corollary}

To prove Theorem \ref{thm:main}, we need the following three lemmas, all of which
are (slight generalizations of) some special cases of Theorem \ref{thm:main}:
\begin{lemma}[``3--column Lemma''] \label{lem:3} 
Let $G$ be an orthogonal group, and $(a  > b = c \geq d)$ be even, non-negative integers with $a = \infty$. Suppose
$(a' > b' > c' \geq d')$ is obtained from $(a > b = c \geq d)$ by a fundamental
degeneration (Definition \ref{def:funddegeneration}). Then
$$I^G\left(\la[c,d]^+\boxtimes \mathrm{T}(b| \pm 1)\right) \approx I^G\left(\la[c,b]^-\boxtimes \mathrm{T}(d| \mp 1) \right) \oplus I^G\left(\la[c',d']^{+}\boxtimes \mathrm{T}(b'| \pm 1)\right).$$
%
\end{lemma}
\begin{proof}
Consider the orthogonal orbit $\C O = (b,c,d)$. Then by Theorem \ref{thm:2},
$$\Gamma(\C O) = I^G\left(\la[c,d]^+\boxtimes \mathrm{T}(b| +1)\right) \approx \C B(\ovl{\C O})$$
Let $\C P := (b',c',d')$ so that $Norm(\C O) = \{\C O,\C P\}$, then the candidates of composition factors of $\C B(\ovl{\C O})$ are
\begin{equation} \label{eq:mod3}
\begin{aligned}
&\C M(\C O,\theta) \quad (\approx I^G\left(\la[c,b]^-\boxtimes \mathrm{T}(d | \theta)\right))  \quad &&\text{where}\ \theta \in \C A(\C O)^{\vee};\\ 
&\C M(\C P, \chi) \quad (\approx I^G\left(\la[c',d']^{\chi_1}\boxtimes \mathrm{T}(b'| \chi_2)\right)) \quad &&\text{where}\ \chi \in \C A(\C P)^{\vee}
\end{aligned}
\end{equation}
(the $\approx$ above comes from Remark \ref{rmk:parastructure}). Then the result follows by comparing the diminutive $K_c-$type multiplicities of $\Gamma(\C O) \approx \C B(\ovl{\C O})$ and the modules in \eqref{eq:mod3}. The proof is similar for $I^G(\la[c,d]^+\boxtimes \mathrm{T}(b| -1))$.
\end{proof}

\begin{example}
Let $\C O = (6,6,2)$. Then $\C P = (8,4,2)$ and we would like to study the composition factors of
$$I^G(\la[6,2]^{+} \boxtimes \mathrm{T}(6| \pm 1)).$$
The modules in the brackets of \eqref{eq:mod3} are
\begin{equation} 
I^G\left(\la[6,6]^-\boxtimes\ \mathrm{T}(2| \pm 1)\right) \quad \text{and} \quad I^G\left(\la[4,2]^{\pm}\boxtimes\ \mathrm{T}(8| \pm 1)\right),
\end{equation}
By Frobenius reciprocity, one can calculate the diminutive $K_c-$type multiplicities of these modules:
\begin{center}
  \begin{tabular}{ | c || c | c | c | c | c | c | c | c | c | c |}
    \hline
     & $\wedge^0\mathbb{C}^{14}$ & $\wedge^2\mathbb{C}^{14}$ & $\wedge^4\mathbb{C}^{14}$ & $\wedge^6\mathbb{C}^{14}$ & $\wedge^8\mathbb{C}^{14}$ & $\wedge^{10}\mathbb{C}^{14}$ & $\wedge^{12}\mathbb{C}^{14}$ & $\wedge^{14}\mathbb{C}^{14}$ \\ \hline 
    $I^G(\la[6,2]^+\boxtimes\mathrm{T}(6|+1))$ & $1$ & $1$ & $1$ & $1$ & $1$ & $0$ & $0$ & $0$ \\  
		$I^G(\la[6,2]^+\boxtimes\mathrm{T}(6|-1))$ & $0$ & $0$ & $0$ & $1$ & $1$ & $1$ & $1$ & $1$ \\ \hline  \hline
	
		$I^G(\la[6,6]^-\boxtimes\mathrm{T}(2|+1))$ & $0$ & $0$ & $0$ & $1$ & $0$ & $0$ & $0$ & $0$  \\ \hline
   $I^G\left(\la[6,6]^-\boxtimes\mathrm{T}(2| -1)\right)$ & $0$ & $0$ & $0$ & $0$ & $1$ & $0$ & $0$ & $0$  \\ \hline
	$I^G\left(\la[4,2]^+\boxtimes\mathrm{T}(8| +1)\right)$ & $1$ & $1$ & $1$ & $1$ & $0$ & $0$ & $0$ & $0$ \\ \hline
	$I^G\left(\la[4,2]^+\boxtimes\mathrm{T}(8| -1)\right)$ & $0$ & $0$ & $0$ & $0$ & $1$ & $1$ & $1$ & $1$ \\ \hline
	$I^G\left(\la[4,2]^-\boxtimes\mathrm{T}(8| +1)\right)$ & $0$ & $1$ & $1$ & $0$ & $0$ & $0$ & $0$ & $0$ \\ \hline
	$I^G\left(\la[4,2]^-\boxtimes\mathrm{T}(8| -1)\right)$ & $0$ & $0$ & $0$ & $0$ & $0$ & $1$ & $1$ & $0$ \\ \hline
  \end{tabular}
\end{center}
It is easy to see that 
$$I^G\left(\la[6,2]^+\boxtimes\mathrm{T}(6| +1)\right) \approx 
I^G\left(\la[6,6]^-\boxtimes\mathrm{T}(2| -1)\right) \oplus I^G\left(\la[4,2]^+\boxtimes\mathrm{T}(8| +1)\right)$$
and
$$I^G\left(\la[6,2]^+\boxtimes\mathrm{T}(6| -1)\right) \approx 
I^G\left(\la[6,6]^-\boxtimes\mathrm{T}(2| +1)\right) \oplus I^G\left(\la[4,2]^+\boxtimes\mathrm{T}(8| -1)\right).$$
\end{example}

\begin{lemma}[``4--column Lemma''] \label{lem:4} 
Let $G$ be a symplectic group, and $a \geq b = c \geq d$ be even, non-negative integers. Suppose
$a' \geq b' \geq c' \geq d'$ is obtained from $a \geq b = c \geq d$ by a fundamental
degeneration (Definition \ref{def:funddegeneration}). Then
$$I^G\left(\la[a,b]^+\boxtimes \la[c,d]^+\right) \approx I^G\left(\la[c,b]^-\boxtimes \la[a,d]^{\sigma_1} \right) \oplus I^G\left(\la[a',b']^{+}\boxtimes[c',d']^{+}\right),$$
where $\sigma_{1} := -$ if $d > 0$, or $+$ if $d =0$.

Moreover, if $a > b$, then the composition factors of 
$$I^G\left(\la[a,b]^-\boxtimes \la[c,d]^+\right) \approx I^G\left(\la[c,b]^-\boxtimes \la[a,d]^{\sigma_2} \right) \oplus I^G\left(\la[a',b']^{-}\boxtimes[c',d']^{+}\right),$$
where $\sigma_{2} := +$ if $d > 0$, or $+$ if $d =0$. 
\end{lemma}
\begin{proof}
Consider $\C O = (a,b,c,d)$. Then
$$\Gamma(\C O) = I^G(\la[a,b]^+ \boxtimes \la[c,d]^+) \approx \C B(\ovl{\C O}).$$
As in the previous Lemma, $Norm(\C O) = \{\C O, \C P := (a',b',c',d')\}$ and the candidates of the composition factors of $\C B(\ovl{\C O})$ are
\begin{equation} \label{eq:mod4}
\begin{aligned}
&\C M(\C O,\theta) \quad (\approx I^G\left(\la[c,b]^-\boxtimes\la[a,d]^{\theta}\right)), &&\text{where}\ \theta \in \ovl{A}(\C O)^{\vee};\\
&\C M(\C P,\chi) \quad (\approx I^G\left(\la[a',b']^{\chi_1}\boxtimes\la[c',d']^{\chi_2}\right)) &&\text{where}\ \chi \in \ovl{A}(\C P)^{\vee}.
\end{aligned}
\end{equation}
Then the result follows by comparing the diminutive $K_c-$type multiplicities of $\Gamma(\C O) \approx \C B(\ovl{\C O})$ and the modules in \eqref{eq:mod4}.
\end{proof}

\begin{example}
Let $\C O = (6,2,2,2)$ be non-generic orbit in $\mathfrak{sp}(12,\mathbb{C})$. Then $\C P = (6,4,1,1)$ and we study the composition factors of 
$$I^G(\la[6,2]^{\pm}\boxtimes\la[2,2]^+)$$
The modules in \eqref{eq:mod4} are
\begin{equation} 
I^G\left(\la[2,2]^-\boxtimes\la[6,2]^{\pm}\right) \quad \text{and} \quad I^G\left(\la[6,4]^{\pm}\boxtimes\la[1,1]^+\right),
\end{equation}
As before, one can calculate the diminutive $K_c-$type multiplicities of these modules:
\begin{center}
  \begin{tabular}{ | c || c | c | c | c |}
    \hline
     & $\wedge^0\mathbb{C}^{12}$ & $\wedge^2\mathbb{C}^{12}/\wedge^0\mathbb{C}^{12}$ & $\wedge^4\mathbb{C}^{12}/\wedge^2\mathbb{C}^{12}$ & $\wedge^6\mathbb{C}^{12}/\wedge^4\mathbb{C}^{12}$ \\ \hline
    $I^G(\la[6,2]^+\boxtimes\la[2,2]^+)$ & $1$ & $1$ & $1$ & $0$  \\
		$I^G(\la[6,2]^-\boxtimes\la[2,2]^+)$ & $0$ & $1$ & $1$ & $1$  \\  \hline  \hline 
   $I^G\left(\la[2,2]^-\boxtimes\la[6,2]^+\right)$ & $0$ & $1$ & $0$ & $0$  \\ \hline
	$I^G\left(\la[2,2]^-\boxtimes\la[6,2]^-\right)$ & $0$ & $0$ & $1$ & $0$ \\ \hline
	$I^G\left(\la[6,4]^+\boxtimes[1,1]^+\right)$ & $1$ & $1$ & $0$ & $0$  \\ \hline
	$I^G\left(\la[6,4]^-\boxtimes[1,1]^+\right)$ & $0$ & $0$ & $1$ & $1$ \\ \hline
  \end{tabular}
\end{center}
Once again, one can see easily that 
$$I^G\left(\la[6,2]^+\boxtimes\la[2,2]^+\right) \approx I^G\left(\la[2,2]^-\boxtimes\la[6,2]^-\right) \oplus I^G\left(\la[6,4]^+\boxtimes[1,1]^+\right)$$
and
$$I^G\left(\la[6,2]^-\boxtimes\la[2,2]^+\right) \approx I^G\left(\la[2,2]^-\boxtimes\la[6,2]^+\right) \oplus I^G\left(\la[6,4]^-\boxtimes[1,1]^+\right).$$
\end{example}

For the purpose of the induction arguments below, we record the following Lemma, whose proof is identical
to that of Lemma \ref{lem:4}:
\begin{lemma}[``5--column Lemma''] \label{lem:5} 
Let $G$ be an orthogonal group, and $z > a \geq b = c \geq d$ be even, non-negative integers. Suppose
$a' \geq b' \geq c' \geq d'$ is obtained from $a \geq b = c \geq d$ by a fundamental
degeneration (Definition \ref{def:funddegeneration}). Then
\begin{align*} &I^G\left(\la[a,b]^+\boxtimes \la[c,d]^+\boxtimes\ \mathrm{T}(z| +1)\right) \\
\approx\ &I^G\left(\la[c,b]^-\boxtimes \la[a,d]^{\sigma_1}\boxtimes\ \mathrm{T}(z| +1) \right) \oplus I^G\left(\la[a',b']^{+}\boxtimes[c',d']^{+}\boxtimes\ \mathrm{T}(z| +1)\right),
\end{align*}
where $\sigma_{1} := -$ if $d > 0$, or $+$ if $d =0$.

Moreover, if $a > b$, then the composition factors of 
\begin{align*} &I^G\left(\la[a,b]^-\boxtimes \la[c,d]^+\boxtimes\ \mathrm{T}(z| +1)\right) \\
\approx\ & I^G\left(\la[c,b]^-\boxtimes \la[a,d]^{\sigma_2}\boxtimes\ \mathrm{T}(z| +1) \right) \oplus I^G\left(\la[a',b']^{-}\boxtimes[c',d']^{+}\boxtimes\ \mathrm{T}(z| +1)\right),
\end{align*}
where $\sigma_{2} := +$ if $d > 0$, or $+$ if $d =0$. 
%
\end{lemma}

\bigskip
\noindent{\it Proof of Theorem \ref{thm:main}.} We now prove Theorem \ref{thm:main} by induction on the length of 
$$\C O = \begin{cases} (c_0, c_1, \dots, c_{2p}, c_{2p+1}) & \text{if}\ G = Sp(2n,\mathbb{C}),\\
(c_1, c_2, \dots, c_{2p}, c_{2p+1}) & \text{if}\ G = O(2n,\mathbb{C})
\end{cases}.$$ 

\subsection{Initial Step} Let $p = 1$ and $\C O_1 = (c_1,c_2,c_3)$ be an orthogonal orbit. Then
$$Norm(\C O_1) = \begin{cases} \{\C O_1\} & \text{if}\ c_1 \neq c_2;\\
\{\C O_1, \C P_1\} & \text{if}\ c_1 = c_2
\end{cases}$$
where $\C P_1 := (c_1',c_2',c_3')$ is obtained from $\C O_1$ by a fundamental degeneration (with $c_0 = \infty$). 
The distinguished modules are
$$\begin{cases}
\{\C M(\C O_1, \ep(\C O_1))\} & \text{if}\ c_1 \neq c_2; \\ 
\left\{\C M(\C O_1, \ep(\C O_1)) , \ \ \C M(\C P_1, \ep(\C P_1)) \right\} & \text{if}\ c_1 = c_2. 
\end{cases}$$
By Remark \ref{rmk:parastructure}, the sum of the distinguished modules are $\approx$ to
$$\begin{cases}
I^G(\la[c_2,c_3]^+\boxtimes\mathrm{T}(c_1|+1)) & \text{if}\ c_1 \neq c_2; \\ 
I^G(\la[c_2,c_1]^-\boxtimes\mathrm{T}(c_3|-1)) \oplus I^G(\la[c_2',c_3']^+\boxtimes\mathrm{T}(c_1'|+1)) & \text{if}\ c_1 = c_2. 
\end{cases}$$
In both cases, the sum is $I^G\left(\la[c_2,c_3]^+\boxtimes \mathrm{T}(c_1|+1)\right) = \Gamma(\C O_1)$ by the 3-column lemma (Lemma \ref{lem:3}). 

\medskip
Let $\C O_1 = (c_0,c_1,c_2,c_3)$ be a symplectic orbit. Then
$$Norm(\C O_1) = \begin{cases} \{\C O_1\} & \text{if}\ c_1 \neq c_2;\\
\{\C O_1, \C P_1\} & \text{if}\ c_1 = c_2
\end{cases}$$
where $\C P_1 = (c_0',c_1',c_2',c_3')$ is obtained from $\C O_1$ by a fundamental degeneration. 
As before, the sum of the distinguished modules are $\approx$ to
$$\begin{cases}
I^G(\la[c_0,c_1]^+|\la[c_2,c_3]^+ & \text{if}\ c_1 \neq c_2; \\ 
I^G(\la[c_2,c_1]^-|\la[c_0,c_3]^-) \oplus I^G(\la[c_0',c_1']^+|\la[c_2',c_3]^+) & \text{if}\ c_1 = c_2. 
\end{cases}$$
In both cases, the sum is $I^G\left(\la[c_0,c_1]^+\boxtimes \la[c_2,c_3]^+\right) = \Gamma(\C O_1)$ by the 4-column lemma (Lemma \ref{lem:4}). 

\subsection{Induction Step} Now suppose the proposition holds for $p = k-1$. We give a 
complete proof of the case of $p = k$ for orthogonal orbits, while the arguments for symplectic 
orbits is similar.

\medskip
Let $\C O_{k} = (c_1,c_2, \dots, c_{2k},c_{2k+1})$ be an orthogonal orbit with
\begin{equation} \label{eq:ok-1}
\Gamma(\C O_{k}) = I^G\left(\la[c_2,c_3]^+\boxtimes\cdots\boxtimes\la[c_{2k},c_{2k+1}]^+\boxtimes\ \mathrm{T}(c_1|+)\right) \approx 
\bigoplus_{\C P_{k} \in Norm(\C O_{k})} \C M(\C P_{k}, \ep(\C P_{k})).
\end{equation}
It is useful to keep in mind that the following holds:
\begin{itemize}
\item[(i)] $\C P_{k} = (d_1, d_2, \dots, d_{2s},d_{2s+1})$;
\item[(ii)] $\tau_0(\C P_{k}) = \{i_1, \dots, i_r\}$;
\item[(iii)] $\C P_{k}^* = (d_1^*, d_2^*, \dots, d_{2t}^*,d_{2t+1}^*)$;
\item[(iv)] $c_{2k+1} = d_{2s+1} = d_{2t+1}^*$;
\item[(v)] $\C A(\C P_{k}) = (\mathbb{Z}/2\mathbb{Z})^{a_0} \times  \dots \times (\mathbb{Z}/2\mathbb{Z})^{a_{t}}$.
\end{itemize}
By Remark \ref{rmk:parastructure}, $\C M(\C P_{k};\ep(\C P_{k}))$ satisfies
$$
 \C M(\C P_{k}, \ep(\C P_{k})) \approx 
 I^G\begin{pmatrix} \la[d_{2i_1},d_{2i_1-1}]^-\boxtimes \cdots \boxtimes \la[d_{2i_r},d_{2i_r-1}]^-\boxtimes  \quad\quad \quad \quad \quad \quad \quad\quad \quad \quad \\ 
\quad \quad\quad   \la[d_2^*,d_3^*]^{\ep(\C P_{k})_1} \boxtimes \cdots \boxtimes \la[d_{2t}^*,d_{2t+1}^*]^{\ep(\C P_{k})_{t}} \boxtimes\ \mathrm{T}(d_1^*| \ep(\C P_{k})_0)\end{pmatrix}
$$

Now consider $\C O_{k+1} = (c_1,c_2, \dots, c_{2k},c_{2k+1},c_{2k+2},c_{2k+3})$. By \eqref{eq:ok-1},
$$\Gamma(\C O_{k+1}) = I^G\left(\la[c_{2k+2},c_{2k+3}]^+\boxtimes\ \Gamma(\C O_{k})\right) \approx 
\bigoplus_{\C P_{k} \in Norm(\C O_{k})} I^G(\la[c_{2k+2},c_{2k+3}]^+\boxtimes\  \C M(\C P_{k}, \ep(\C P_{k})))$$
So we wish to study the sum on the right. There are two cases:

\medskip
\noindent \underline{(i) $c_{2k-1} \neq c_{2k}$:} For each  $\C P_{k} \in Norm(\C O_{k})$,
let
$$\C P_{k+1} := (d_1, d_2, \dots, d_{2s}, d_{2s+1}, c_{2k+2}, c_{2k+3}).$$
Then it is easy to see from the definition of $Norm(\C O)$ that
\[
Norm(\C O_{k+1}) = \bigcup_{\C P_{k} \in Norm(\C O_{k})} \{ \C P_{k+1}\}.
\]
We now look at $\C P_{k+1}$ more closely: $\tau_0(\C P_{k+1}) = \tau_0(\C P_{k})$ since $d_{2s+1} = c_{2k+1} \neq c_{2k+2}$. Therefore,
\begin{equation} \label{eq:pkstar}
\C P_{k+1}^* = ( d_1^*, d_2^*, \dots, d_{2t}^*, d_{2t+1}^*,c_{2k+2},c_{2k+3}), \quad \ep(\C P_{k+1}) = \left(\ep(\C P_{k})_{0}, \dots, \ep(\C P_{k})_{t}, +\right)
\end{equation}
Now 
\begin{align*}
&I^{G}(\la[c_{2k+2},c_{2k+3}]^+ \ \boxtimes\  \C M(\C P_{k}, \ep(\C P_{k}))) && \\
\approx\ &I^G \begin{pmatrix} \la[c_{2k+2},c_{2k+3}]^+ \boxtimes \la[d_{2i_1},d_{2i_1-1}]^-\boxtimes \cdots \boxtimes \la[d_{2i_r},d_{2i_r-1}]^- \boxtimes \quad\quad \quad \quad\quad \quad \\ 
\quad\quad \quad \quad\quad \quad \quad\quad  \la[d_2^*,d_3^*]^{\ep(\C P_{k})_1} \boxtimes \cdots \boxtimes  \la[d_{2t}^*,d_{2t+1}^*]^{\ep(\C P_{k})_{t}} \boxtimes\ \mathrm{T}(d_1^*| \ep(\C P_{k})_0) \end{pmatrix}\\
\approx\ &I^G\begin{pmatrix}\la[d_{2i_1},d_{2i_1-1}]^-\boxtimes \cdots \boxtimes \la[d_{2i_r},d_{2i_r-1}]^-\boxtimes \quad\quad \quad\quad\quad \quad \quad\quad \quad\quad\quad \quad\quad\quad \quad \quad\quad \quad \quad\quad \quad\\
\quad\quad \quad\quad \quad\quad   \la[d_2^*,d_3^*]^{\ep(\C P_k)_1} \boxtimes \cdots \boxtimes \la[d_{2t}^*,d_{2t+1}^*]^{\ep(\C P_{k})_{t}} \boxtimes \la[c_{2k+2},c_{2k+3}]^+ \boxtimes\ \mathrm{T}(d_1^*| \ep(\C P_{k})_0) \end{pmatrix} \\
\approx\ &\C M(\C P_{k+1},\ep(\C P_{k+1})).
\end{align*}
Here the last $\approx$ comes from Remark \ref{rmk:parastructure} and \eqref{eq:pkstar}.

\bigskip
\noindent \underline{(ii) $c_{2k-1} = c_{2k}$:} For each  $\C P_{k} \in Norm(\C O_{k})$,
let
$$\C P_{k+1}^u := (d_1, d_2, \dots, d_{2s}, d_{2s+1}, c_{2k+2}, c_{2k+3}),$$
with $d_{2s+1} = c_{2k+1} = c_{2k+2}$. So we also let
$$\C P_{k+1}^d := (d_1, d_2, \dots, d_{2s}', d_{2s+1}', c_{2k+2}', c_{2k+3}').$$
by performing a fundamental degeneration on $d_{2s} \geq d_{2s+1} = c_{2k+2} \geq c_{2k+3}$. Then
\[
Norm(\C O_{k+1}) = \bigcup_{\C P_{k} \in Norm(\C O_{k})} \{ \C P_{k+1}^u, \C P_{k+1}^d\}.
\]
As before, we look at $\C P_{k+1}^u$ and $\C P_{k+1}^d$ more closely: Firstly, $\tau_0(\C P_{k+1}^u) = \tau_0(\C P_{k}) \cup \{k+1\}$ since $d_{2s+1} = c_{2k+1} = c_{2k+2}$. Therefore,
\begin{equation} \label{eq:pkustar}
(\C P_{k}^u)^* = ( d_1^*, d_2^*, \dots, d_{2t}^*,c_{2k+3}), \quad \ep(\C P_{k+1}^u) = \left(\ep(\C P_{k})_{0}, \dots, -\ep(\C P_{k})_{t}\right).
\end{equation}
On the other hand, $\tau_0(\C P_{k+1}^d) = \tau_0(\C P_{k})$ since $d_{2s+1}' \neq c_{2k+2}'$, and hence
\begin{equation} \label{eq:pkdstar}
(\C P_{k}^d)^* = ( d_1^*, d_2^*, \dots, d_{2t}^*,d_{2s+1}',c_{2k+2}', c_{2k+3}'), \quad \ep(\C P_{k+1}^d) = \left(\ep(\C P_{k})_{0}, \dots, \ep(\C P_{k})_{t}, +\right).
\end{equation}

Consider
\begin{equation} \label{eq:updown2}
\begin{aligned}
&I^{G}(\la[c_{2k+2},c_{2k+3}]^+ \boxtimes\ \C M(\C P_{k}, \ep(\C P_{k})))  \\
\approx\ &I^G\begin{pmatrix} \la[c_{2k+2},c_{2k+3}]^+ \boxtimes \la[d_{2i_1},d_{2i_1-1}]^-\boxtimes \cdots \boxtimes \la[d_{2i_r},d_{2i_r-1}]^-  \boxtimes \quad\quad \quad \quad\quad \quad \\
\quad\quad \quad \quad\quad \quad\quad \quad\la[d_2^*,d_3^*]^{\ep(\C P_k)_1} \boxtimes \cdots \boxtimes \la[d_{2t}^*,d_{2t+1}^*]^{\ep(\C P_{k})_{t}} \boxtimes\ \mathrm{T}(d_1^*| \ep(\C P_{k})_0) \end{pmatrix}\\
\approx\ &I^G\begin{pmatrix}\la[d_{2i_1},d_{2i_1-1}]^-\boxtimes \cdots \boxtimes \la[d_{2i_r},d_{2i_r-1}]^-\boxtimes \quad\quad \quad\quad\quad \quad\quad\quad \quad \quad\quad \quad\quad\quad \quad \quad\quad \quad\\
\quad\quad \quad \la[d_2^*,d_3^*]^{\ep(\C P_k)_1} \boxtimes \cdots  \boxtimes \la[d_{2t}^*,d_{2t+1}^*]^{\ep(\C P_{k})_{t}} \boxtimes \la[c_{2k+2},c_{2k+3}]^+ \boxtimes\ \mathrm{T}(d_1^*| \ep(\C P_{k})_0) \end{pmatrix}
\end{aligned}
\end{equation}
Look at the last part of the last module in \eqref{eq:updown2}. If $t = 0$, then $d_1^*= d_{2t+1}^* = c_{2k+1} = c_{2k+2}$. By applying the 3-column lemma (Lemma \ref{lem:3}) with $(d_1^* = c_{2k+2} \geq c_{2k+3})$, the last module of \eqref{eq:updown2} is $\approx$ to
\begin{align*}
&I^G\left( \cdots \boxtimes \la[c_{2k+2},d_1^*]^- \boxtimes\ \mathrm{T}(c_{2k+3}| -\ep(\C P_{k})_1) \right) \oplus
I^G\left( \cdots \boxtimes \la[c_{2k+2}',c_{2k+3}']^+ \boxtimes\ \mathrm{T}(d_1'| \ep(\C P_{k})_1)\right)\\
\approx\ &\C M(\C P^u_{k+1};\ep(\C P^u_{k+1})) \oplus \C M(\C P^d_{k+1},\ep(\C P^d_{k+1}))
\end{align*}
where the last $\approx$ follows from \eqref{eq:pkustar} and \eqref{eq:pkdstar}.

\medskip
On the other hand, if $t > 0$, then $d_{2t+1}^* = c_{2k+1} = c_{2k+2}$. By applying the 5-column lemma (Lemma \ref{lem:5}) with $(d_1^* > d_{2t}^* \geq d_{2t+1}^* = c_{2k+2} \geq c_{2k+3})$, the last module is $\approx$ to
\begin{align*}
&I^G\left( \cdots \boxtimes \la[c_{2k+2},d_{2t+1}^*]^- \boxtimes \la[d_{2t}^*,c_{2k+3}]^{-\ep(\C P_{k})_{t}} \boxtimes\ \mathrm{T}(d_1^*| \ep(\C P_{k})_0)\right) \oplus \\
&I^G\left( \cdots \boxtimes \la[d_{2t}',d_{2t+1}']^{\ep(\C P_{k})_{t}} \boxtimes \la[c_{2k+2}',c_{2k+3}']^+ \boxtimes\ \mathrm{T}(d_1^*| \ep(\C P_{k})_0) \right)
\end{align*}
which is also $\approx$ to $\C M(\C P_{k+1}^u, \ep(\C P_{k+1}^u)) \oplus \C M(\C P_{k+1}^d, \ep(\C P_{k+1}^d))$.

\bigskip
By summing up all $\C P_{k} \in Norm(\C O_{k})$ in all the cases above, one always have
\begin{align*}
&\ \ \bigoplus_{\C P_{k+1} \in Norm(\C O_{k+1})} \C M(\C P_{k+1}, \ep(\C P_{k+1})) \\
&\approx\ 
\bigoplus_{\C P_{k} \in Norm(\C O_{k})} I^{G}(\la[c_{2k+2},c_{2k+3}]^+ \boxtimes\ \C M(\C P_{k}, \ep(\C P_{k})))\\
& \approx\  I^{G}(\la[c_{2k+2},c_{2k+3}]^+\boxtimes\la[c_2,c_3]^+\boxtimes\cdots\boxtimes\la[c_{2k},c_{2k+1}]^+ \boxtimes\ \mathrm{T}(c_1|+))\\
& \approx\  \Gamma(\C O_{k+1})
\end{align*}
for all even, generic orthogonal orbits. The proof for symplectic orbits is similar (and easier) using the 4-column lemma (Lemma \ref{lem:4}). Therefore, Theorem \ref{thm:main} follows. \qed

\bigskip
As in \cite{BW1}, the same arguments apply to odd orbits, and for general orbits involving 
both odd and even columns. We skip the proofs there.

\bigskip
Finally, we discuss how one can obtain the character formula of $\C B(\ovl{\C O})$ for all classical nilpotent orbits.
By Theorem \ref{thm:reduce}, the character formula of $\C B(\ovl{\C O})$ can be obtained from that of
$\C B(\ovl{\C O_{gen}})$. Using \eqref{eq:parameter} and the character formulas of unipotent representations given in \cite{BV2},
one can get the character formulas for the distinguished modules $\C M(\C P;\ep(\C P))$ for all $\C P \in Norm(\C O_{gen})$.
Consequently, the character formula of $\C B(\ovl{\C O_{gen}})$ can be obtained by summing up the formulas 
$\displaystyle \C B(\ovl{\C O_{gen}}) \cong \bigoplus_{\C P \in Norm(\C O_{gen})} \C M(\C P;\ep(\C P)).$
As a consequence, one can calculate the multiplicities of {\bf all} $K_c-$types in $\C B(\ovl{\C O}) \cong R(\ovl{\C O})$.
Explicit calculations are carried out in \cite{Wo1}. 

We end this paper by an example of
the generic orbit $\C O = (4,2,2,0)$ in $G = Sp(8,\mathbb{C})$ (with non-normal closure):

\begin{example}
Let $\C O = (4,2,2,0)$ be a generic orbit in $G = Sp(8,\mathbb{C})$. The infinitesimal character for $\C B(\ovl{\C O})$ is 
$\la_{\C O} = (2,1,1,0)$. The character formula of $\C B(\ovl{\C O})$ is given by:
\begin{align*}
\frac{1}{2}\left[ \sum_{w_1 \in C_2 \times D_1 \times C_1}\mathrm{sgn}(w_1) X(
(21, 0, 1);w_1(21, 0, 1)) + \sum_{w_2 \in D_3 \times C_1}\mathrm{sgn}(w_2) X(
(210, 1);w_2(210, 1))\right] 
\end{align*}
Upon restricting to $K$, the formula can be rewritten as:
\begin{align*}
\C B(\ovl{\C O}) \cong &\mathrm{Ind}_T^K(\mathbb{C}_{0000})
- 2\mathrm{Ind}_T^K(\mathbb{C}_{1100})
-\mathrm{Ind}_T^K(\mathbb{C}_{2000})
+4\mathrm{Ind}_T^K(\mathbb{C}_{2110}) \\
-2&\mathrm{Ind}_T^K(\mathbb{C}_{2211})
-\mathrm{Ind}_T^K(\mathbb{C}_{2200})
+\mathrm{Ind}_T^K(\mathbb{C}_{2222})
+2\mathrm{Ind}_T^K(\mathbb{C}_{3210})\\
-2&\mathrm{Ind}_T^K(\mathbb{C}_{3221})
-\mathrm{Ind}_T^K(\mathbb{C}_{4110})
+\mathrm{Ind}_T^K(\mathbb{C}_{4211})
+\mathrm{Ind}_T^K(\mathbb{C}_{4200})
-\mathrm{Ind}_T^K(\mathbb{C}_{4220}).
\end{align*}
On the other hand, the character formula of $\C B(\C O)$ in \eqref{eq:sunip} 
can be written as:
\begin{align*}
\sum_{w \in C_2 \times A_1}\mathrm{sgn}(w_1) X(
(21, 10);w(21, 10))
\end{align*}
whose restriction to $K$ is given by
\begin{align*}
\C B(\C O) \cong\   &\mathrm{Ind}_T^K(\mathbb{C}_{0000})
- 2\mathrm{Ind}_T^K(\mathbb{C}_{1100})
+\mathrm{Ind}_T^K(\mathbb{C}_{1111})
-\mathrm{Ind}_T^K(\mathbb{C}_{2000}) \\
+&\mathrm{Ind}_T^K(\mathbb{C}_{2110})
+2\mathrm{Ind}_T^K(\mathbb{C}_{3100})
-2\mathrm{Ind}_T^K(\mathbb{C}_{3111})
-\mathrm{Ind}_T^K(\mathbb{C}_{3300})\\
+&\mathrm{Ind}_T^K(\mathbb{C}_{3311})
-\mathrm{Ind}_T^K(\mathbb{C}_{4000})
+\mathrm{Ind}_T^K(\mathbb{C}_{4110})
+\mathrm{Ind}_T^K(\mathbb{C}_{4200})
-\mathrm{Ind}_T^K(\mathbb{C}_{4211}).
\end{align*}
The difference between the formulas occurs at $\mathbb{C}_{1111}$. This implies
the diminutive $K_c-$type $V_{(1111)}$ has different multiplicities in $R(\ovl{\C O})$ and in $R(\C O)$.
More precisely, $Norm(\C O) = \{\C O, (4,4)\}$, and the difference between $R(\ovl{\C O})$ and $R(\C O)$ is
exactly equal to
$$\C M(44,-)|_K = \mathrm{Ind}_{GL(4)}^{Sp(8)}(\det)|_K \cong  \sum_{w \in A_3}\mathrm{sgn}(w) \mathrm{Ind}_T^K(
(210-1)-w(10-1-2)).$$
\end{example}

\end{document}